\documentclass[12pt]{amsart}
\usepackage{amsmath,amssymb,amscd}

\newtheorem{propo}{Proposition}[section]
\newtheorem{corol}[propo]{Corollary}
\newtheorem{theor}[propo]{Theorem}
\newtheorem{lemma}[propo]{Lemma}

\theoremstyle{definition}
\newtheorem{defin}[propo]{Definition}
\newtheorem{examp}[propo]{Example}

\theoremstyle{remark}
\newtheorem{remar}[propo]{Remark}

\newcommand{\vp}{\varphi}
\newcommand{\mb}{\mathbb}

\newcommand{\ot}{\otimes}

\newcommand{\bean}{\begin{eqnarray}}
\newcommand{\eean}{\end{eqnarray}}
\newcommand{\bea}{\begin{eqnarray*}}
\newcommand{\eea}{\end{eqnarray*}}
\newcommand{\bsa}{\begin{subarray}{c}}
\newcommand{\esa}{\end{subarray}}
\newcommand{\bi}{\begin{itemize}}
\newcommand{\ei}{\end{itemize}}
\newcommand{\ben}{\begin{enumerate}}
\newcommand{\een}{\end{enumerate}}

\newcommand {\CC} {\mathbb{C}}

\newcommand {\ZZ} {\mathbb{Z}}
\newcommand {\NN} {\mathbb{N}}
\newcommand{\id}{\mathrm{id}}
\newcommand{\BAR}{\overline}

\newcommand{\ii}{\mathrm{i}}
\newcommand{\cnst}{\mathrm{c}}
\DeclareMathOperator{\SL}{SL}
\DeclareMathOperator{\diag}{diag}

\newcommand{\defst}[1]{{\it #1}}

\numberwithin{equation}{section}

\title[$V_L^+$ and type $D^{(1)}$ level $2$]
{An isomorphism between the fusion algebras of $V_L^+$ and type $D^{(1)}$
level $2$}

\author{M.~Cuntz}
\address{Michael Cuntz,
Fachbereich Mathematik,
Universit\"at Kaiserslau\-tern,
Postfach 3049,
D-67653 Kaiserslautern, Germany}
\email{cuntz@mathematik.uni-kl.de}

\author{C.~Goff}
\address{Christopher Goff,
Mathematics Department,
University of the Pacific,
Stockton, CA  95211, USA}
\email{cgoff@pacific.edu}

\begin{document}

\begin{abstract}
The fusion algebra of the vertex operator algebra $V_L^+$ for a rank $1$ even lattice
$L$ is explicitly shown to be isomorphic to the fusion algebra of the Kac-Moody algebra of type
$D^{(1)}$ at level $2$ in almost all cases.
\end{abstract}

\maketitle

\section{Introduction}

In this article, we prove:

\vspace{3mm}

\noindent \textbf{Theorem \ref{mainthm}.}\hspace{-3mm}  \emph{
Let $L = \sqrt{2\ell}\, \ZZ$ for some $\ell \in \NN$. Then the fusion algebra for the vertex operator
algebra $V_L^+$ is isomorphic to the fusion algebra
\[
\left\{ \begin{array}{cl} D^{(1)}_{\ell,2} & \text{if $\ell \ge 3$} \\
A^{(1)}_{1,2} \ot A^{(1)}_{1,2} & \text{if $\ell = 2$} \\
A^{(1)}_{7,1} \cong \ZZ[\ZZ/8\ZZ] & \text{if $\ell = 1$}.   \end{array} \right.
\] 
(Recall that $D^{(1)}_{3,2} \cong A^{(1)}_{3,2}$.)
}

\vspace{3mm}

The main part of the proof consists of an explicit computation of the
$S$-matrix of a Kac-Moody algebra of type $D^{(1)}$ at level $2$. This enables us to find the structure constants directly, avoiding the use of
algorithms such as the one given in \cite{aFsF}.
To obtain the isomorphism, we compare our structure constants
with those given for the vertex operator algebra $V_L^+$ in \cite{tAcDhL}.
As a consequence, we also obtain a bijection between the sets of
irreducible representations.

This work has appeared elsewhere in the physics literature.  Therefore our strictly computational proof serves only as further mathematical confirmation for this well-known physical result.

The paper is organized as follows: section 2 deals with preliminaries about fusion algebras and trigonometric (mostly cosine) identities, a surprising range of which were very useful to our proof.  We also describe the Kac-Peterson matrix of type $X_{\ell}^{(1)}$ or $A^{(2)}_{2\ell}$ and give the explicit formula we use (\ref{kpmat}).  Section 3 specializes this discussion to type $D$, with the important Example~\ref{exD2} of $D^{(1)}_{\ell}$ at level 2 (denoted $D^{(1)}_{\ell,2}$).  The notation we set up there is used throughout the rest of the paper.  Section \ref{extpow} discusses exterior powers: Proposition~\ref{S=detR} gives a variety of instances in which the Kac-Peterson formula reduces to the determinant of a certain matrix.  In section 5, the lengthiest and most computational section, we work out each entry in the $S$-matrix for $D_{\ell,2}^{(1)}$ and explicitly give the character table, which we call the $s$-matrix.
The fusion rules of the vertex operator algebra (VOA) $V_L^+$ arising from a rank 1 lattice $L$ are listed in section 6. Proofs for these rules can be found in \cite{tA} or \cite{tAcDhL}.  Finally, section 7 contains the proof of the main theorem.

\section{Preliminaries} \label{seccos}

\subsection{Fusion algebras}
We give here a short introduction to the notion of fusion algebras;
for more details, see for example \cite{mC1}.
Fusion algebras are algebras with the same properties as
Grothendieck rings of semisimple tensor categories. So, for example,
character rings of finite groups are fusion algebras. One should
imagine them as lattices with an algebra structure:

\begin{defin}
Let $R$ be a finitely generated commutative $\ZZ$-algebra
which is a free $\ZZ$-module with basis
$B=\{b_0=1_R,\ldots,b_{n-1}\}$ and \defst{structure constants}
\[ b_i b_j = \sum_{k=0}^{n-1} N_{ij}^k b_k, \quad N_{ij}^k\in \ZZ_{\ge 0} \]
for $0\le i,j < n$.
Assume that there is an involution $\sim : R \rightarrow R$
which is a $\ZZ$-module homomorphism such that
\[ \tilde B= B,\quad N_{\tilde i\tilde j}^k=N_{i j}^{\tilde k},
\quad N_{\tilde ij}^0 = \delta_{i,j} \]
for all $0\le i,j,k < n$,
where $\tilde i$ is the index with $\tilde{b_i}=b_{\tilde i}$.
Then we call $(R,B)$ a \defst{fusion algebra}.
\end{defin}

A fusion algebra $R$ has many nice properties. The most important
one is that $R\otimes_\ZZ \CC$ is semisimple which comes from the
fact that it is a symmetric algebra with respect to $\sim$.
Since this is a commutative finite dimensional algebra over $\CC$, there is
a base change matrix to the algebra $\CC^n$. We denote such a base
change by $s$ and call it an \defst{$s$-matrix of $R$}. This matrix is the
\defst{character table} of $R$, because its rows correspond
to the irreducible characters of $R$: The entries of a row are the
values of a representation of $R$ at the basis elements $B$.
This also implies that the rows of $s$ are orthogonal.

A fusion algebra is completely determined by its $s$-matrix.
Permuting rows does not change the algebra; permuting columns
changes the ordering of $B$.
\label{moddat}
If a fusion algebra is the Grothendieck ring of a modular tensor
category, then what we call its $s$-matrix is almost its $S$-matrix.
In this case one has also a $T$-matrix; the matrices $(S,T)$
define a representation of the modular group and are therefore
called a \emph{modular datum} (see for example \cite{tG2} or \cite{mC2}
for an exact definition).
It is then possible to arrange the rows and columns of $s$ in such a way that
the matrix $(s_{i,j}s_{0,i})_{i,j}$ is symmetric and
\[ s_{i,j} = \frac{S_{i,j}}{S_{i,0}}. \]
The matrix $S$, if it exists, is not uniquely determined by $s$.
Nevertheless, if $R$ comes from a modular datum $(S,T)$ then we call $S$
the \defst{$S$-matrix of $R$}. It is easy to see in this case that the above structure
constants are given by the formula (Verlinde's formula)
\begin{equation}\label{verl}
N_{ij}^m:=\sum_{k=0}^{n-1} \frac{S_{k,i}S_{k,j}\BAR{S_{k,m}}}{S_{k,0}}
\end{equation}
for all $0\le i,j,m \le n-1$.
This formula is nothing more than the fact that $R$ is the lattice
spanned by the columns of $s$ with componentwise multiplication.

\subsection{Cosine Identities and Roots of Unity}

We fix $\ii$ to be a square root of $-1$ and define $\zeta_n:=\exp\left(\frac{2\pi \ii}{n}\right)$ for $n \in \mb{N}$. 

Details for our proof require some well-known cosine identities which we reproduce here.  Let $\ell \in \mb{N}$ and let $x\notin 2\pi\ZZ$.  Then the following identities can be verified using geometric series of complex exponentials.
{\small
\bean 
1 + 2\cos x + 2\cos 2x + \ldots + 2\cos \ell x & = & \frac{\sin(\frac{(2\ell +1)x}{2})}{\sin(\frac{x}{2})} \label{DK} \\
\hspace{1cm} \cos(\vp) + \cos(\vp + x) + \ldots + \cos(\vp + \ell x) & = & \frac{\sin(\frac{(\ell+1)x}{2})\cos(\vp + \frac{\ell x}{2})}{\sin(\frac{x}{2})}
\label{cosphi}
\eean
}

Define 
\[
c(j):=2\cos \left(\frac{2\pi j}{2\ell}\right)\ \textrm{and}\ s(j):=2\sin \left(\frac{2\pi j}{2\ell}\right).
\]
Note that for $j \in \ZZ$, $\zeta_{2\ell}^j + \zeta_{2\ell}^{-j} = c(j)$ and $\zeta_{2\ell}^j - \zeta_{2\ell}^{-j} = \ii\, s(j)$.

We immediately obtain several trivial identities, such as:
$c(\alpha)c(\beta) = c(\alpha + \beta) + c(\alpha - \beta)$ (for all $\alpha,\beta$) and $c(\ell m + j) = c(\ell m - j) = (-1)^m c(j)$ (for $m \in \ZZ$), which we use in the sequel, among others.

Define $\rho_j := \frac{1}{1+\delta_{0,j} + \delta_{\ell,j}}$ for $0 \leq j \leq \ell$ (and note that $\rho_j = \rho_{\ell-j}$).  Also, let $\rho_{i,j} := \frac{\rho_i\rho_j}{2\ell}$.

\begin{lemma} \label{lemcos}  Let $m \in \mb{Z}$ with $0 \leq m \leq 2\ell$.  Then the following identities hold.
\bean
\label{rhocij} \sum_{j=0}^\ell \rho_j c(mj) & = & \left\{ \begin{array}{cl} 
0 & \text{if}\ m \notin \{0,2\ell\} \\
2\ell & \text{if}\ m \in \{0,2\ell\} \\
\end{array} \right. \\
 \label{oddcos} \sum_{j=1}^\ell c((2j-1)m/2) & = & \left\{ \begin{array}{cl} 0 & \text{if}\ m \notin \{0,2\ell\} \\
2\ell & \text{if}\ m = 0 \\
-2\ell & \text{if}\ m = 2\ell
\end{array} \right.
\eean
\end{lemma}
\begin{proof}  To see (\ref{rhocij}), first note that if $m \notin \{0,\ell,2\ell\}$, then (\ref{rhocij}) follows from (\ref{DK}) and other trigonometric identities.  When $m \in \{0,2\ell\}$, we have $\sum_j 2\rho_j = 2\ell$.  When $m = \ell$, we have $\sum_j 2(-1)^j \rho_j = 1 - 2 + 2 - \ldots + (-1)^\ell = 0$.

For (\ref{oddcos}), the special cases $m \in \{0, 2\ell\}$ are clear.  Let $m \notin \{0, 2\ell\}$.  Using (\ref{cosphi}) we obtain
$\sum_{j=1}^\ell c((2j-1)m/2) = \frac{s(\ell m/2)c(\ell m/2)}{s(m/2)} = \frac{s(\ell m)}{s(m/2)} = 0$.
\end{proof}

The following corollary lists certain matrices that arise in section~\ref{D2}, along with some properties and their determinants (up to a sign, which will be settled in section~\ref{D2}).  Each equation follows from a cosine identity, such as (\ref{oddcos}), and so we leave proofs to the reader.  The lemma then relates two of these matrices.
\begin{corol} As before, let $\ell \in \NN$.
\bi
\item \emph{(The $M$ Matrix)}
Let $M$ and $N$ be $(\ell + 1) \times (\ell + 1)$ matrices with $M_{i,j} := c(ij)$ and $N_{i,j}:= \rho_{i,j} c(ij)$ for $0 \leq i,j \leq \ell$. Then 
\begin{equation} \label{mmat}
N = M^{-1}\quad \textrm{and so}\quad (\det M)^2 = 16(2\ell)^{\ell+1}.
\end{equation}

\item \emph{(The $X$ Matrix)}
Let $X$ be an $\ell \times \ell$ matrix with $X_{i,j}:=c\left(\frac{(2i-1)(2j-1)}{4}\right)$ for $1 \leq i,j \leq \ell$.  Then
\begin{equation} \label{xmat}
X^{2} = 2\ell I_\ell\quad \textrm{and so}\quad (\det X)^2 = (2\ell)^\ell.
\end{equation}

\item \emph{(The $\Omega$ Matrix)}
Let $\Omega$ be an $\ell \times \ell$ matrix with $\Omega_{i,j}:=c\left(\frac{(i-1)(2j-1)}{2}\right)$ for $1 \leq i,j \leq \ell$.  Then 
\begin{equation} \label{Omat}
(\det \Omega)^2 = 2(2\ell)^{\ell}.
\end{equation}
\ei
\end{corol}

\begin{lemma}\label{OXmat}
$\det X = \frac{1}{\sqrt{2}} \det \Omega$, where $\sqrt{2}$ is the \emph{positive} square root
of $2$.
\end{lemma}
\begin{proof}
Let $\xi:=\zeta_{8\ell}$, so $X_{i,j}=\xi^{(2i-1)(2j-1)}+\xi^{-(2i-1)(2j-1)}$.
Dividing each $j$-th column of $X$ by $(\xi^{(2j-1)}+\xi^{-(2j-1)})$ divides
the determinant by $\alpha:=\prod_{j=1}^\ell(\xi^{(2j-1)}+\xi^{-(2j-1)})$, hence
\begin{eqnarray*}
\det X &=& \alpha \det \left(\xi^{(2-2i)(2j-1)}+\xi^{(4-2i)(2j-1)}+\ldots
+\xi^{(2i-2)(2j-1)}\right)_{i,j} \\
&\stackrel{(*)}{=}&
\alpha \det \left(\xi^{(2-2i)(2j-1)}+\xi^{(2i-2)(2j-1)}\right)_{i,j} \\
&=& \alpha \det \left(c\left(\frac{(i-1)(2j-1)}{2}\right) \right)_{i,j}=
\alpha \det \Omega
\end{eqnarray*}
where $(*)$ denotes subtracting row $i$ from row $i+1$ for $i=1,\ldots,\ell-1$. Now
\[ \alpha = \prod_{j=1}^\ell(\xi^{(2j-1)}+\xi^{-(2j-1)})
= \prod_{j=1}^\ell \cos \frac{2\pi (2j-1)}{8\ell}>0 \]
and thus $\alpha=\frac{1}{\sqrt{2}}$ by (\ref{xmat}) and (\ref{Omat}).
\end{proof}

\subsection{Kac-Peterson matrices}

Let ${\frak g}(A)$ be an affine Kac-Moody algebra corresponding to an
$n\times n$ generalized Cartan matrix $A$ of rank $\ell$, with Cartan subalgebra $\frak h$ and with
$\langle \:,\: \rangle : {\frak h}\times{\frak h^*} \rightarrow \CC$
as its corresponding pairing (we use the notation of \cite{vK}). Define
\[ P:=\{\lambda\in {\frak h^*} \mid \langle \lambda,\alpha_i^\vee\rangle \in \ZZ,\quad i=0,\ldots,n-1 \}, \]
\[ P_+:=\{\lambda\in P \mid \langle \lambda,\alpha_i^\vee\rangle \ge 0,\quad
i=0,\ldots,n-1 \}. \]
The set $P$ is called the \defst{weight lattice}, and
elements of $P$ (resp. $P_+$) are called
\defst{integral} (resp. \defst{dominant integral}) \defst{weights}.

Now let ${\frak g}(A)$ be of arbitrary untwisted type $X_\ell^{(1)}$
or $A_{2\ell}^{(2)}$ (then $n=\ell+1$).
The {\it fundamental weights} $\Lambda_i \in P$,
$i=0,\ldots,\ell$ are given by the equations
\[ \langle \Lambda_i,\alpha_j^\vee\rangle = \delta_{ij}, \quad \langle \Lambda_i,d\rangle=0\]
for $j=0,\ldots,\ell$, where $d\in {\frak h}$ is given
by $\langle \alpha_i,d\rangle=\delta_{i,0}$.
The set $\{ \alpha_0^\vee,\ldots,\alpha_\ell^\vee,d \}$ is
a basis of $\frak h$ and
$\{ \alpha_0,\ldots,\alpha_\ell,\Lambda_0 \}$ is a basis of $\frak h^*$.
The fundamental weights $\BAR\Lambda_i$ of the
finite dimensional Lie algebra $\frak g^\circ$ satisfy
\[ \Lambda_i = \BAR \Lambda_i + a_i^\vee\Lambda_0, \]
($\BAR \Lambda_0=0$ because $a_0^\vee=1$; the vector of $a_i^\vee$ is
an element of the kernel of $A$, see \cite{vK}, 6.1).

For each positive integer $k$, let
$P_+^k\subseteq P_+$ be the finite set
\begin{equation} \label{P+k} 
P_+^k := \Big\{ \sum_{j=0}^\ell \lambda_j\Lambda_j\mid
\lambda_j\in\ZZ_{\geq 0}, \sum_{j=0}^\ell a_j^\vee\lambda_j=k \Big\}. \end{equation}
Kac and Peterson defined a natural $\CC$-representation of the group $\SL_2(\ZZ)$
on the subspace spanned by the affine characters of ${\frak g}$ which are indexed
by $P_+^k$.
The image of $\tiny \begin{pmatrix} 0&-1\\ 1&0 \end{pmatrix}$ under
this representation is the so-called {\it Kac-Peterson matrix}.  (See Theorem $13.8$ of \cite{vK}.)
For affine algebras of type $X_\ell^{(1)}$ or $A_{2\ell}^{(2)}$, this matrix is
\begin{equation}\label{kpmat}
S_{\Lambda,\Lambda'} = \cnst \sum_{w\in W^\circ} \det(w)
\exp\left({-\, \frac{2\pi \ii (\BAR\Lambda+\bar\rho\mid
 w(\BAR\Lambda'+\bar \rho))}{k+h^\vee}}\right),
\end{equation}
where $\Lambda,\Lambda'$ runs through $P_+^k$, $(\cdot\mid\cdot)$ is
the normalized bilinear form of chapter $6$ of \cite{vK}, $\BAR{\rho} = \sum_{j=1}^\ell \BAR{\Lambda_j}$ and $W^\circ$ is the Weyl group of $\frak g^\circ$.
The constant $\cnst$ is a normalization factor such that $\BAR{S}^T S=\id$
and such that the first column and row have positive real numbers as entries.

A Kac-Peterson matrix of type $X_\ell^{(1)}$ and level $k$ defines a fusion
algebra which we denote by $X_{\ell,k}^{(1)}$. A classification
of all these fusion algebras up to isomorphism was given by Gannon in \cite{tG}.

\section{The structure of $P_+^k$ for type $D$}

Using $\Lambda_i=\BAR{\Lambda_i}+a^\vee_i\Lambda_0$, we obtain
\[ P_+^k = \Big\{ k\Lambda_0 + \sum_{j=1}^\ell \lambda_j \BAR{\Lambda_j}
\mid \lambda_j\in \ZZ_{\ge 0}, \:\: \sum_{j=1}^\ell a_j^\vee \lambda_j\le k\Big\}. \]
In the Kac-Peterson formula the $\Lambda_0$ component is not used, so
we will ignore this component in the rest of the paper.
Moreover, in (\ref{kpmat}) we always add $\BAR\rho$ to the weights.
We will therefore identify the indices for $S$ with the set
\[ P^k_\rho:= \Big\{ \BAR\rho+ \sum_{j=1}^\ell \lambda_j \BAR{\Lambda_j}
\mid \lambda_j\in \ZZ_{\ge 0}, \:\: \sum_{j=1}^\ell a_j^\vee \lambda_j\le k\Big\}. \]

We now specialize to a type $D$ affine Lie algebra.  Let $\ell \ge 3$.  (Even though $D_3 \cong A_3$ the construction below still works.)  First we fix a basis $\{v_i\ |\ 1 \leq i \leq \ell \}$ of ${\frak h}^{\circ*}$,
which is orthonormal with respect to $(\cdot\mid\cdot)$ and such that
\[ \alpha_i = \begin{cases} v_i - v_{i+1} & 1 \leq i \leq \ell-1 \\
v_{\ell-1} + v_\ell & i = \ell \end{cases} \]
is the usual root system (see \cite[6.7]{vK}). 
The fundamental weights $\BAR{\Lambda_j}$ are dual to the $\alpha_i^\vee$ for $1 \leq i,j \leq \ell$. Explicitly, then,
\[ \BAR{\Lambda_j} = \begin{cases} v_1 + v_2 + \ldots + v_j & 1 \leq j \leq \ell-2 \\
\frac12(v_1+v_2 + \ldots +v_{\ell-1} - v_\ell) & j=\ell-1 \\
\frac12(v_1+v_2 + \ldots +v_{\ell-1} + v_\ell) & j=\ell. \end{cases} \]
From the kernel of the Cartan matrix of $D^{(1)}_\ell$ come the numbers 
\[ a_i^\vee = \begin{cases} 2 & 2 \leq i \leq \ell-2 \\
1 & i=1,\ell-1,\ell. \end{cases} \]
Further, $\overline{\rho} = (\ell-1)v_1 + (\ell-2)v_2 + \ldots + v_{\ell-1}$.  

The following example contains notation used throughout the rest of the paper.

\begin{examp} \textbf{$D^{(1)}_{\ell,2}$} \label{exD2}

\textbf{Warning:} at this point, and in section \ref{D2}, we write the elements of $P_\rho^2$ in reverse order with respect to the basis $\{v_i\ |\ 1 \leq i \leq \ell\}$, which will simplify the calculations later.  Thus, $\overline{\rho} = v_{\ell-1} + 2v_{\ell-2} + \ldots + (\ell-2)v_2 + (\ell-1)v_1 = (0,1,2, \ldots, \ell-2, \ell-1)$, etc.  We begin with $\ell$ vectors that are of the form $(0,1,2,\ldots,j-1,j+1, \ldots, \ell-1, \ell)$ with $j \neq 0$ and then write down the last seven elements of $P^2_\rho$.
\[
\begin{array}{rclcl}
\overline{\rho} & = & \left(0,1,2, \ldots, \ell-2, \ell-1\right) &=:& \nu_0 \\
\overline{\rho} + \overline{\Lambda_1} & = & \left(0,1,2, \ldots, \ell-3, \ell-2, \ell\right) &=:& \nu_1\\
\vdots && \hspace{2cm}\vdots && \vdots \\
\overline{\rho} + \overline{\Lambda_{\ell-1}} + \overline{\Lambda_\ell} & = & \left(0,2,3, \ldots, \ell-1, \ell\right) &=:& \nu_{\ell-1} \\
\overline{\rho} + 2\overline{\Lambda_1} & = & \left(0,1,2, \ldots, \ell-2, \ell+1\right) &=:& \nu'_0\\
\overline{\rho} + 2\overline{\Lambda_{\ell-1}} & = & \left(-1,2,3, \ldots, \ell-1, \ell\right) &=:& \nu'_\ell  \\
\overline{\rho} + 2\overline{\Lambda_\ell} & = & \left(1,2,3, \ldots, \ell-1, \ell\right) &=:& \nu_\ell\\
\overline{\rho} + \overline{\Lambda_\ell} & = & \frac12 \left(1, 3, 5, \ldots, 2\ell-3, 2\ell-1\right) &=:& \mu_0 \\
\overline{\rho} + \overline{\Lambda_{\ell-1}} & = & \frac12 \left(-1, 3, 5, \ldots, 2\ell-3, 2\ell-1\right) &=:& \mu_1  \\
\overline{\rho} + \overline{\Lambda_1} + \overline{\Lambda_\ell} & = & \frac12 \left(1, 3, 5, \ldots, 2\ell-3, 2\ell+1\right) &=:& \mu_2  \\
\overline{\rho} + \overline{\Lambda_1} + \overline{\Lambda_{\ell-1}} & = & \frac12 \left(-1, 3, 5, \ldots, 2\ell-3, 2\ell+1\right) &=:& \mu_3
\end{array}\]
Note that there are $\ell+7$ irreducible representations for $D^{(1)}_{\ell,2}$.  They are of two types, which we denote $Z$ and $H$ according to whether their entries are in $\ZZ$ or $\frac12\ZZ \setminus \ZZ$, respectively. So 
\[ Z = \{\nu'_0, \nu'_\ell, \nu_i\ |\ 0 \leq i \leq \ell \}\ \text{and}\ H = \{\mu_i\ |\ 0 \leq i \leq 3\}. \]
\end{examp}

\section{Exterior powers and type $D$}\label{extpow}

There is an obvious similarity between the formula for the
Kac-Peterson matrix and a determinant. Indeed, for example
the matrices of type $C$ may be viewed as exterior powers of
matrices of type $A_1$ (see \cite{mC1}). In this article we use an
analogous observation for type $D$.

Let $\ell\in\NN$, $e:=k+h^\vee$ and $\zeta:=\zeta_e$.
For any $\lambda,\mu \in P$, let $R^{\lambda,\mu}$ be the matrix
\[ R^{\lambda,\mu}_{i,j} := \zeta^{\lambda_i\mu_j}+\zeta^{-\lambda_i\mu_j}. \]
where $\lambda = \sum_{i=1}^\ell \lambda_i v_i$, $\mu = \sum_{i=1}^\ell \mu_i v_i$.

Since we use the basis of $v_i$'s, the Weyl group of type
$D$ is the group of monomial matrices with $\pm 1$ as non-zero entries
and with an even number of $-1$'s (see \cite[6.7]{vK}).
Throughout all proofs, we will denote by $\Xi\subset\{\pm 1\}^\ell$ the set of
vectors with entries $\pm 1$ and an even number of $-1$'s.

\begin{propo} \label{S=detR}
Let $S$ be the Kac-Peterson matrix of type $D$, level $k$ and
$\lambda,\mu \in P_\rho^k$ with $\lambda_i,\mu_i\in\ZZ$
for all $1\le i\le \ell$. Then
\[ S_{\lambda,\mu} = \frac{\cnst}{2} \det R^{\lambda, \mu} \]
if $e \mid 2\lambda_i$ or $e \mid 2\mu_i$ for some $i$.
\end{propo}

\begin{proof}
By equation (\ref{kpmat}),
\[ S_{\lambda,\mu}=\cnst \sum_{\sigma\in S_\ell}\sum_{f\in \Xi}
\varepsilon_\sigma \exp\left({-\frac{2\pi \ii (\lambda \mid
 \sigma(\mu)^f)}{k+h^\vee}}\right) \]
where $\varepsilon_\sigma$ is the sign of the permutation $\sigma$
and $\mu^f=(f_1\mu_1,\ldots,f_\ell\mu_\ell)$.
Assume one of the above conditions is
satisfied, say without loss of generality that $e \mid 2\lambda_1$. We obtain
\begin{eqnarray*}
\frac{1}{\cnst}S_{\lambda,\mu} & = &
\sum_{\sigma\in S_\ell}\sum_{f\in\Xi}
\varepsilon_\sigma \zeta^{-(\lambda \mid \sigma(\mu)^f)} \\
& = & \sum_{f\in\Xi} \sum_{\sigma\in S_\ell}
\varepsilon_\sigma \prod_{i=1}^{\ell} \zeta^{-f_i \lambda_i \mu_{\sigma(i)}} \\
& = & \sum_{f\in\Xi} \sum_{\sigma\in S_\ell}
\varepsilon_\sigma (\zeta^{\lambda_1})^{(-f_1) \mu_{\sigma(1)}}
\prod_{i=2}^{\ell} \zeta^{-f_i \lambda_i \mu_{\sigma(i)}}.
\end{eqnarray*}
Since $\zeta^{\lambda_1} = \pm 1$ and $\mu_j \in \ZZ$, it is irrelevant whether $f_1 = 1$ or $-1$.  Hence, the last expression becomes (notice the $\frac12$ factor and the new summation index)
\begin{eqnarray*}
& &\frac{1}{2} \sum_{f\in\{\pm 1\}^\ell} \sum_{\sigma\in S_\ell}
\varepsilon_\sigma (\zeta^{\lambda_1})^{(-f_1) \mu_{\sigma(1)}}
\prod_{i=2}^{\ell} \zeta^{-f_i \lambda_i \mu_{\sigma(i)}} \\
& = &\frac{1}{2} \sum_{\sigma\in S_\ell} \varepsilon_\sigma
\sum_{f\in\{\pm 1\}^\ell}
\prod_{i=1}^{\ell} \zeta^{-f_i \lambda_i \mu_{\sigma(i)}} \\
& = &\frac{1}{2} \sum_{\sigma\in S_\ell} \varepsilon_\sigma
\prod_{i=1}^{\ell} (\zeta^{\lambda_i \mu_{\sigma(i)}}+
\zeta^{-\lambda_i \mu_{\sigma(i)}}) \\
& = &\frac{1}{2} \det R^{\lambda, \mu}.
\end{eqnarray*}
\end{proof}

\begin{remar} \label{S=detR2}
Proposition \ref{S=detR} holds more generally.  For instance, if $\lambda_i = 0$ for some $i$ and if $\mu_j \in \frac{1}{n}\ZZ$ for all $j$, then an analogous calculation is still valid, only with $\zeta_{ne}^n$ in place of $\zeta_e$ to guarantee integral exponents.
\end{remar}

\section{Type $D$ level $2$: The $S$-matrix}\label{D2}

From now on, we concentrate on the case of type $D$ and level $2$.  (See Example \ref{exD2}.) Note $k+h^\vee=2\ell$. 
Most of the labels for $P_\rho^2$ are of the form
$\lambda = \nu_{i}$ for some $i$.
In this notation, $\overline{\rho} = \nu_0$ is the unit of the fusion algebra.
Further, the constant in (\ref{kpmat}) is
\[ \cnst = \ii^{|\Delta_+^\circ|}|\Gamma^*/(k+h^\vee)\Gamma|^{-\frac12} \]
where $\Gamma$ is the lattice defined in \cite[6.5]{vK} (called $M$ there) and $\Delta_+^\circ$ is the set of positive roots.
Since
$|\Delta_+^\circ|=\ell^2-\ell$ for type $D_\ell$ (see for example \cite[12.2]{jH}),
$\cnst$ is a real number and its sign is $(-1)^{\binom{\ell}{2}}$.
By \cite[13.8.10]{vK}, all entries in the first column and first row of $S$
are positive real numbers: $S_{\nu_0,\lambda}=S_{\lambda,\nu_0}>0$ for all $\lambda$.
We will use this fact several times to determine needed signs.
We will also frequently use
$(-1)^{\binom{\ell}{2}}=(-1)^{\lfloor\frac{\ell}{2}\rfloor}$.

Most of the section is devoted to the calculation of $S_{\lambda,\mu}$ in all instances, but broadly, we devote a subsection to each case: $\lambda, \mu \in Z$; $\lambda \in Z, \mu \in H$; and $\lambda, \mu \in H$. 
In the last subsection we summarize the complete $s$-matrix, the $S$-matrix and the $T$-matrix.

\subsection{$\lambda, \mu \in Z$}
Note that Proposition~\ref{S=detR} applies here.
Hence we already have information regarding the large part of the $S$-matrix (with size depending
on $\ell$). 
\begin{propo}  Let $0 \le i,j \le \ell$.  Then
\[ S_{\nu_i,\nu_j} = \cnst (-1)^{\binom{\ell}{2}}
2\sqrt{2^{\ell+1}\ell^{\ell+1}} \rho_{i,j} c(ij). \]
\end{propo}
\begin{proof}
Remember from (\ref{mmat}) the matrix $M=(c(ij))_{0\le i,j\le \ell}$ and its inverse.
Notice that $R^{\nu_i,\nu_j}$ is the submatrix of $M$ where
we remove the $(\ell-i)$-th row and $(\ell-j)$-th column. Hence, for all $i$ and $j$,
\[ \rho_{\ell-j,\ell-i} c((\ell-j)(\ell-i)) = (M^{-1})_{\ell-j,\ell-i}
= (-1)^{i+j} \frac{\det R^{\nu_i,\nu_j}}{\det M}. \]
So, we have 
\bea
S_{\nu_i,\nu_j} = \frac{\cnst}{2} \det R^{\nu_i,\nu_j} & = & \frac{\cnst}{2}(-1)^{i+j}(\det M)\rho_{\ell-j,\ell-i}c((\ell-j)(\ell-i)) \\
& = & \frac{\cnst}{2}(-1)^{i+j}(\det M)\rho_{i,j}c(\ell(\ell-j-i)+ij) \\
& = & \frac{\cnst}{2}(-1)^{\ell}(\det M)\rho_{i,j}c(ij). \\
\eea
Specifically, we have $0<S_{\nu_0,\nu_0} = \frac{(-1)^\ell \cnst}{8\ell} (\det M)$.
Hence the sign of $\det M$ is
$(-1)^{\binom{\ell}{2}+\ell}$
and by (\ref{mmat}), $\det M=(-1)^{\binom{\ell}{2}+\ell} 4\sqrt{2^{\ell+1}\ell^{\ell+1}}$,
and so the proposition holds.
\end{proof}

The corresponding entry in the $s$-matrix follows directly:
\[
s_{\nu_i,\nu_j} = \frac{S_{\nu_i,\nu_j}}{S_{\nu_i, \nu_0}} = \frac{\rho_{i,j}c(ij)}{\rho_{i,0}c(0)}  = \rho_j c(ij).
\]

For the columns and rows indexed by $\nu'_0$ and $\nu'_\ell$, we use:
\begin{lemma}
For all $\lambda \in Z$, we have $R^{\nu'_0,\lambda} = R^{\nu_0,\lambda}$ and $R^{\nu'_\ell,\lambda} = R^{\nu_\ell,\lambda}$.
\end{lemma}
\begin{proof}  Notice that $\nu_0$ and $\nu'_0$ differ only in the $\ell$-th coordinate, while $\nu_\ell$ and $\nu'_\ell$ differ only in the first coordinate.  So $R^{\nu'_0,\lambda}$ and $R^{\nu_0,\lambda}$ possibly differ only in the $\ell$-th row, $R^{\nu'_\ell,\lambda}$ and $R^{\nu_\ell,\lambda}$ possibly only in the first.  But since $\lambda_j \in \ZZ$ for all $j$, we get $R^{\nu'_0,\lambda}_{\ell,j} = c((\ell+1)\lambda_j) = c((\ell-1)\lambda_j) = R^{\nu_0,\lambda}_{\ell,j}$.  Also, $R^{\nu'_\ell,\lambda}_{1,j} = R^{\nu_\ell,\lambda}_{1,j}$ because cosine is an even function.
\end{proof}

\subsection{$\lambda\in Z, \mu \in H$}
Many of the entries of the $S$-matrix for $\lambda\in Z$, $\mu \in H$ are
still given by determinants.  (See Remark~\ref{S=detR2}.) 
Notice that when $\lambda_\ell = \ell$ (that is, for all $\lambda \in Z \setminus \{\nu_0,\nu'_0\}$), then $\det R^{\lambda,\mu} = 0$ for all $\mu \in H$ because $R^{\lambda,\mu}_{\ell,j} = c(\ell \mu_j) = 0$ for all $j$.
One immediate consequence of this paragraph is that, for all $\mu \in H$,
\[
S_{\nu_i,\mu} = 0\ \ \text{for}\ \ 1 \leq i \leq \ell-1.
\]
This is not true for $S_{\nu_\ell,\mu}$ or $S_{\nu'_\ell,\mu}$ because
Remark~\ref{S=detR2} does not apply, but we have:

\begin{lemma} For all $\mu \in H$, $S_{\nu_\ell,\mu} = - S_{\nu'_\ell,\mu}$.
\end{lemma}
\begin{proof}  Define $\omega := \zeta_{4\ell}$.
\bea
\frac{1}{\cnst} S_{\nu_\ell,\mu} + \frac{1}{\cnst}S_{\nu'_\ell,\mu} & = & \sum_{f\in\Xi} \sum_{\sigma\in S_\ell}
\varepsilon_\sigma \left[ \prod_{i=1}^{\ell} \omega^{-f_i (\nu_\ell)_i 2\mu_{\sigma(i)}} + \prod_{i=1}^{\ell} \omega^{-f_i (\nu'_\ell)_i 2\mu_{\sigma(i)}} \right] \\
& = &\sum_{f\in\Xi} \sum_{\sigma\in S_\ell}
\varepsilon_\sigma \left[\omega^{-f_1 2\mu_{\sigma(1)}} + \omega^{+f_1 2\mu_{\sigma(1)}} \right]\prod_{i=2}^{\ell} \omega^{-f_i (\nu_\ell)_i 2\mu_{\sigma(i)}} \\
& = &\sum_{f\in \{\pm1\}^\ell} \sum_{\sigma\in S_\ell}
\varepsilon_\sigma \left[\omega^{-f_1 2\mu_{\sigma(1)}} \right]\prod_{i=2}^{\ell} \omega^{-f_i (\nu_\ell)_i 2\mu_{\sigma(i)}} \\
& = & \det R^{\nu_\ell,\mu} = 0.
\eea
\end{proof}

Moreover, we have:
\begin{lemma} For all $\mu \in H$, $S_{\nu_0,\mu} = -S_{\nu'_0,\mu}$.
\end{lemma}
\begin{proof}
As stated earlier, for $\lambda \in \{\nu_0,\nu'_0 \}$,
we have $\frac{1}{\cnst}S_{\lambda,\mu} = \frac12 \det R^{\lambda,\mu}$.  
Comparing the matrix entries for $R^{\nu_0,\mu}$ and $R^{\nu'_0,\mu}$, we see that the only difference occurs in the $\ell$-th row.  Namely, since $x = 2\mu_j$ is an odd number for all $j$, 
$ R^{\nu'_0,\mu}_{\ell,j} = c((\ell+1)\frac{x}{2}) = -c((\ell-1)\frac{x}{2}) = -R^{\nu_0,\mu}_{\ell,j}$.
\end{proof}

The following two lemmas give explicit values for the $S$-matrix.

\begin{lemma}\label{detomat}
Recall $\Omega$ from (\ref{Omat}).  Then
\[ S_{\nu_0,\mu} = \frac{\cnst}{2} \det \Omega =
\frac{\cnst}{2} (-1)^{\binom{\ell}{2}}\sqrt{2^{\ell+1}\ell^\ell} \]
for all $\mu \in H$.
\end{lemma}
\begin{proof}
Recalling the notation from Example~\ref{exD2}, we have $R^{\nu_0,\mu_0} = R^{\nu_0,\mu_1}$ and $R^{\nu_0,\mu_2} = R^{\nu_0,\mu_3}$ from the fact that cosine is even.  Next, $R^{\nu_0,\mu_2}= R^{\nu_0,\mu_0}$ because 
$ R^{\nu_0,\mu_2}_{i,\ell} = c\left((i-1)(\ell+\frac12)\right)
= c\left((i-1)(\ell-\frac12)\right) = R^{\nu_0,\mu_0}_{i,\ell}$ and all other columns are explicitly identical.  Moreover, $R^{\nu_0,\mu_0}$ is exactly $\Omega$.  
The sign of $\det\Omega$ now follows from $S_{\nu_0,\mu_0}>0$.
\end{proof}

\begin{lemma}
$S_{\nu_\ell,\mu_0} = S_{\nu_\ell,\mu_3}=-S_{\nu_\ell,\mu_1}=-S_{\nu_\ell,\mu_2}=
\cnst \ii^\ell (-1)^{\lfloor\frac{\ell}{2}\rfloor+\ell} \sqrt{2^{\ell-1}\ell^\ell}$.
\end{lemma}
\begin{proof}
As before, $\omega := \zeta_{4\ell}$.  Because $\omega^{-\ell}=-\omega^\ell$, we have:
\begin{eqnarray*}
S_{\nu_\ell,\mu} &=&
\cnst \sum_{\sigma\in S_\ell} \sum_{f\in\Xi}
\varepsilon_\sigma \prod_{i=1}^\ell \omega^{-i f_{\sigma(i)}2\mu_{\sigma(i)}} \\
 &=& \cnst \sum_{\sigma\in S_\ell} \sum_{f\in\Xi}
\varepsilon_\sigma f_{\sigma(\ell)}\omega^{-\ell 2\mu_{\sigma(\ell)}}
\prod_{i=1}^{\ell-1} \omega^{-i f_{\sigma(i)}2\mu_{\sigma(i)}}.
\end{eqnarray*}
For an $f\in\Xi$, replacing the entry $f_{\sigma(\ell)}$ by $-f_{\sigma(\ell)}$
gives an $\tilde f\in \{\pm 1\}^\ell\backslash\Xi$. Doing this for all $f\in\Xi$
we obtain exactly the set $\{\pm 1\}^\ell\backslash\Xi$.
Since $\prod_{i=1}^\ell \tilde f_i = -1$, we get (notice the $\frac{1}{2}$ factor)
\begin{eqnarray*}
S_{\nu_\ell,\mu} &=& \frac{\cnst}{2}\sum_{\sigma\in S_\ell} \sum_{f\in\{\pm 1\}^\ell}
\varepsilon_\sigma \prod_{i=1}^\ell f_i \omega^{-i f_{\sigma(i)}2\mu_{\sigma(i)}} 
\\
&=& \frac{\cnst}{2}\sum_{\sigma\in S_\ell} \varepsilon_\sigma \prod_{i=1}^\ell
(\omega^{-i 2\mu_{\sigma(i)}}-\omega^{+i 2\mu_{\sigma(i)}}) \\
&=& \frac{\cnst}{2} \det (\omega^{-i 2\mu_j}-\omega^{i 2\mu_j})_{i,j} \\
& = & \frac{\cnst}{2} (-\ii)^\ell \det (s(i\mu_j))_{i,j}.
\end{eqnarray*}
Since only the first entry of $\mu_1$ is different from $\mu_0$, and it differs only by a negative sign (and sine is odd), then $S_{\nu_\ell,\mu_1} = -S_{\nu_\ell,\mu_0}$.  Similarly, $S_{\nu_\ell,\mu_3} = -S_{\nu_\ell,\mu_2}$. Moreover, to compare $S_{\nu_\ell,\mu_2}$ and $S_{\nu_\ell,\mu_0}$, notice $s\left(\frac{i(2\ell+1)}{2}\right) =  -s\left(\frac{i(2\ell-1)}{2}\right)$,
and so $S_{\nu_\ell,\mu_2} = -S_{\mu_\ell,\mu_0}$.

To calculate $S_{\nu_\ell,\mu_0}$, notice first that
$s\left(\frac{i(2j-1)}{2}\right) =
(-1)^{j+1} c\left(\frac{(\ell-i)(2j-1)}{2}\right)$.
Swaping $m:=\lfloor \frac{\ell}{2} \rfloor$ rows, we obtain
\begin{eqnarray*}
\det \left(s \left(\frac{i (2j-1)}{2}\right)\right)_{i,j} &=&
(-1)^m \det \left((-1)^{j+1} c\left(\frac{(i-1)(2j-1)}{2}\right)\right)_{i,j} \\
&=& (-1)^{m+\binom{\ell+1}{2}+\ell} \det \Omega = \det \Omega.
\end{eqnarray*}
Finally by Lemma \ref{detomat},
\[ S_{\nu_\ell,\mu_0} = \frac{\cnst}{2} (-\ii)^\ell
\det\left(s \left(\frac{i (2j-1)}{2}\right)\right)_{i,j}
= \cnst \ii^\ell (-1)^{m+\ell} \sqrt{2^{\ell-1}\ell^\ell},
\]
and the lemma follows.
\end{proof}

\subsection{$\lambda,\mu \in H$}
The last part of the $S$-matrix cannot be obtained by Proposition~\ref{S=detR}.
\begin{propo}\label{HH}
The $4\times 4$ matrix
\[ W_\ell:=(s_{\lambda,\mu})_{\lambda,\mu \in H}=
\frac{1}{\cnst(-1)^{\binom{\ell}{2}} \sqrt{2^{\ell-1} \ell^\ell}}
(S_{\lambda,\mu})_{\lambda,\mu \in H} \]
with respect to the ordering $\{\mu_0,\mu_1,\mu_2,\mu_3\}$, is:
\begin{eqnarray*}
\ell \equiv 1 \:\:(\mbox{\rm mod } 4) : &\quad&
(\zeta_8^{ij})_{i,j\in\{7,1,5,3\}} \\
\ell \equiv 3 \:\:(\mbox{\rm mod } 4) : &\quad&
(\zeta_8^{ij})_{i,j\in\{1,7,3,5\}} \\
\ell \equiv 2 \:\:(\mbox{\rm mod } 4) : &\quad&
{\tiny \begin{pmatrix}
 0 &  \sqrt2 & -\sqrt2 &  0 \\
 \sqrt2 &  0 &  0 & -\sqrt2 \\
-\sqrt2 &  0 &  0 &  \sqrt2 \\
 0 & -\sqrt2 &  \sqrt2 &  0
\end{pmatrix}} \\
\ell \equiv 0 \:\:(\mbox{\rm mod } 4) : &\quad&
{\tiny \begin{pmatrix}
 \sqrt2 &  0 &  0 & -\sqrt2 \\
 0 &  \sqrt2 & -\sqrt2 &  0 \\
 0 & -\sqrt2 &  \sqrt2 &  0 \\
-\sqrt2 &  0 &  0 &  \sqrt2
\end{pmatrix}}
\end{eqnarray*}
\end{propo}

\begin{proof}
Let $\xi=\zeta_{8\ell}$, $\lambda\in H$ and $\alpha_i:=\zeta_4^{-2\lambda_i}$.
Since $2\lambda_i$ is odd for all $i=1,\ldots,\ell$, we have
\[ \alpha_i^{-1}=-\alpha_i,\quad
\prod_{i=1}^\ell \alpha_i^{f_i} = \prod_{i=1}^\ell f_i \alpha_i\]
for $f\in\{\pm 1\}^\ell$.
Remember that $\mu_0 = \frac{1}{2}(1,3,\ldots,2\ell-1)$ with respect
to the $v_i$. Thus
\begin{equation}\label{mu0lambda}
S_{\mu_0,\lambda} =
\cnst \sum_{\sigma\in S_\ell}\sum_{f\in\Xi} \varepsilon_\sigma
\prod_{i=1}^\ell \xi^{-(2i-1) f_{\sigma(i)} 2\lambda_{\sigma(i)}}.
\end{equation}
Using the permutation $\tau \in S_\ell$ given by $\tau(i):=\ell+1-i$, we may write
\[ \xi^{-(2i-1)} = \zeta_4^{-1} \xi^{2\tau(i)-1} \]
for all $i=1,\ldots,\ell$. Therefore
\begin{eqnarray*}
S_{\mu_0,\lambda} & = & \cnst \sum_{\sigma\in S_\ell}\sum_{f\in\Xi} \varepsilon_\sigma
\prod_{i=1}^\ell \xi^{(2i-1) f_{\sigma\tau(i)} 2\lambda_{\sigma\tau(i)}} \zeta_4^{-f_{\sigma\tau(i)}2\lambda_{\sigma\tau(i)}} \\
& = & \cnst \sum_{\sigma\in S_\ell}\sum_{f\in\Xi} \varepsilon_\sigma \varepsilon_\tau
\prod_{i=1}^\ell \xi^{(2i-1) f_{\sigma(i)} 2\lambda_{\sigma(i)}} \zeta_4^{-f_{\sigma(i)}2\lambda_{\sigma(i)}} \\
& = & \cnst \varepsilon_\tau \sum_{\sigma\in S_\ell}\sum_{f\in\Xi} \varepsilon_\sigma
\prod_{i=1}^\ell \xi^{(2i-1) f_{\sigma(i)} 2\lambda_{\sigma(i)}} \prod_{i=1}^\ell f_{\sigma(i)}\alpha_{\sigma(i)} \\
& = & \cnst \left[ \varepsilon_\tau \prod_{i=1}^\ell \alpha_i \right] \sum_{\sigma\in S_\ell}\sum_{f\in\Xi} \varepsilon_\sigma
\prod_{i=1}^\ell \xi^{(2i-1) f_{\sigma(i)} 2\lambda_{\sigma(i)}}  \\
& = & \beta \BAR{S_{\mu_0,\lambda}}
\end{eqnarray*}
where $\beta:=\varepsilon_\tau \prod_{i=1}^\ell \alpha_i$.

Assume first that $\ell$ is odd. Note that in this case,
$\{\pm 1\}^\ell=\Xi\cup-\Xi$.
Further, $\varepsilon_\tau=(-1)^m$ for $\ell=2m+1$.
For a label $\lambda\in H$,
we have $\sum_{i=1}^\ell 2\lambda_i\in\{\ell^2,\ell^2-2,\ell^2+2\}$.
So depending on $\lambda$, either $\beta=\zeta_4$ or $\beta=-\zeta_4$.
Consider now
\begin{eqnarray}
S_{\mu_0,\lambda}+\BAR{S_{\mu_0,\lambda}}
&=& \cnst \sum_{\sigma\in S_\ell}\sum_{f\in\{\pm 1\}^\ell} \varepsilon_\sigma
\prod_{i=1}^\ell \xi^{-(2i-1) f_{\sigma(i)} 2\lambda_{\sigma(i)}} \nonumber \\
&=& \cnst \det (\xi^{(2i-1)2\lambda_j}+\xi^{-(2i-1)2\lambda_j})_{i,j}. \label{ximat}
\end{eqnarray}
In particular for $\lambda=\mu_0$, we get
$S_{\mu_0,\mu_0}+\BAR{S_{\mu_0,\mu_0}} = \cnst (-1)^{\binom{\ell}{2}}
\sqrt{2^\ell \ell^\ell}$
by (\ref{xmat}), Lemma \ref{OXmat} and Lemma \ref{detomat}.
Hence
\[ S_{\mu_0,\mu_0}
= \frac{\cnst (-1)^{\binom{\ell}{2}}\sqrt{2^\ell \ell^\ell}}{1+(-1)^m\zeta_4}
= \cnst \gamma (-1)^{\binom{\ell}{2}}\sqrt{2^{\ell-1} \ell^\ell}, \]
where $\gamma=\frac{\sqrt2}{1+(-1)^m \zeta_4}$ or
\[ \gamma=\begin{cases}\BAR\zeta_8 & \ell \equiv 1 \:\:(\mbox{\rm mod } 4) \\
\zeta_8 & \ell \equiv 3 \:\:(\mbox{\rm mod } 4) \end{cases}.\]
For the remaining three labels, use the following technique:
Consider $\mu_1 := \frac{1}{2}(-1,3,5,\ldots,2\ell-1)$.
The difference between $\mu_1$ and $\mu_0$ is only in the first entry.
Replacing $1$ by $-1$ in the label corresponds to replacing each
$f=(f_1,\ldots,f_\ell)$ by $(-f_1,f_2,\ldots,f_\ell)$ and hence
$\Xi$ by $-\Xi$ since $\ell$ is odd.
Equation (\ref{mu0lambda}) therefore yields
\[ S_{\mu_0,\mu_1} = \BAR{S_{\mu_0,\mu_0}}. \]
Now for $\mu_2 := \frac{1}{2}(1,3,\ldots,2\ell-3,2\ell+1)$. Again, there
is only one difference between $\mu_2$ and $\mu_0$. Since $\xi^{2\ell+1}=-\xi^{-(2\ell-1)}$,
we get
\[ S_{\mu_0,\mu_2} = -\BAR{S_{\mu_0,\mu_0}}. \]
Thus for $\mu_3 := \frac{1}{2}(-1,3,5,\ldots,2\ell-3,2\ell+1)$,
for the same reasons,
\[ S_{\mu_0,\mu_3} = -{S_{\mu_0,\mu_0}}. \]
All other values are obtained in the same way.

Now assume that $\ell$ is even.
Further, $\varepsilon_\tau=(-1)^m$ for $\ell=2m$.
Depending on $\lambda$, $\prod_{i=1}^\ell \alpha_i$ is either $1$
or $-1$. Now $\Xi=-\Xi$ because $\ell$ is even. Hence
$S_{\mu_0,\lambda}=\BAR{S_{\mu_0,\lambda}}$ and
\[ S_{\mu_0,\lambda} (1-(-1)^m(\pm 1)) = 0. \]
When $m$ is odd, this gives the specific value $S_{\mu_0,\mu_0}=0$. Thus
\[ 0 = \cnst \sum_{\sigma\in S_\ell}\sum_{f\in\Xi} \varepsilon_\sigma
\prod_{i=1}^\ell \xi^{-(2i-1) f_{\sigma(i)}(2\sigma(i)-1)} \]
and so if we sum over $\{\pm1\}^\ell$, then the part over $\Xi$
may be ignored:
\begin{eqnarray*}
(-1)^{\binom{\ell}{2}}\sqrt{2^\ell \ell^\ell} &=&
\det (\xi^{(2i-1)(2j-1)}+\xi^{-(2i-1)(2j-1)})_{i,j} \\
&=& \sum_{\sigma\in S_\ell}\sum_{f\in\{\pm 1\}^\ell\backslash \Xi} \varepsilon_\sigma
\prod_{i=1}^\ell \xi^{-(2i-1) f_{\sigma(i)}(2\sigma(i)-1)}.
\end{eqnarray*}
This time, replacing $1$ by $-1$ in the label maps $\Xi$ to $\{\pm1\}^\ell
\setminus \Xi$, hence
\[ (-1)^{\binom{\ell}{2}}\sqrt{2^\ell \ell^\ell} =
\sum_{\sigma\in S_\ell}\sum_{f\in\Xi} \varepsilon_\sigma
\prod_{i=1}^\ell \xi^{-{\mu_1}_i f_{\sigma(i)}(2\sigma(i)-1)}
= \frac{1}{\cnst} S_{\mu_0,\mu_1}. \]
Similarly we get $S_{\mu_0,\mu_2} = -(-1)^{\binom{\ell}{2}}
\cnst \sqrt{2^\ell \ell^\ell}$ and
$S_{\mu_0,\mu_3} = 0$. All other values (also for $m$ even)
are obtained in the same way.
\end{proof}

\subsection{The $s$-matrix}
We compile our results and write them in terms of the $s$-matrix.
\begin{theor}
If we list the columns and rows in the following order: 
$\nu_0, \nu'_0, \nu'_\ell, \nu_\ell,  
\nu_1, \ldots, \nu_{\ell-1}, \mu_0,\ldots,\mu_3$,
then the $s$-matrix for even $\ell$ becomes
\[ {\tiny \left(\begin{array}{rrrr|rrrrr|rrrr}
1 & 1 & 1 & 1 & 2 & 2 & 2 & \ldots & 2 & \sqrt{\ell} & \sqrt{\ell}& \sqrt{\ell}& \sqrt{\ell} \\ 
1 & 1 & 1 & 1 & 2 & 2 & 2 & \ldots & 2 & -\sqrt{\ell} & -\sqrt{\ell}& -\sqrt{\ell}& -\sqrt{\ell} \\ 
1 & 1 & 1 & 1 & -2 & 2 & -2 & \ldots & -2 & -\ii^\ell\sqrt{\ell}& \ii^\ell\sqrt{\ell}& \ii^\ell\sqrt{\ell}& -\ii^\ell\sqrt{\ell} \\ 
1 & 1 & 1 & 1 & -2 & 2 & -2 & \ldots & -2 & \ii^\ell\sqrt{\ell} & -\ii^\ell\sqrt{\ell}& -\ii^\ell\sqrt{\ell}& \ii^\ell\sqrt{\ell} \\ \hline 
1 & 1 & -1 & -1 &  &  &  &  &  & 0 & 0 & 0 & 0 \\
1 & 1 & 1 & 1 &  &  & \vdots &  &  & 0 & 0 & 0 & 0  \\
1 & 1 & -1 & -1 &  & \ldots & c(ij) & \ldots &  & 0 & 0 & 0 & 0 \\
\vdots & \vdots & \vdots & \vdots &  &  & \vdots &  &  & \vdots & \vdots & \vdots & \vdots  \\
1 & 1 & -1 & -1 &  &  &  &  &  & 0 & 0 & 0 & 0 \\ \hline 
1 & -1 & -\ii^\ell & \ii^\ell & 0 & 0 & 0 & \ldots & 0 &  & & & \\ 
1 & -1 & \ii^\ell & -\ii^\ell & 0 & 0 & 0 & \ldots & 0 & & W_{\ell} & & \\ 
1 & -1 & \ii^\ell & -\ii^\ell & 0 & 0 & 0 & \ldots & 0 & & & & \\ 
1 & -1 & -\ii^\ell & \ii^\ell & 0 & 0 & 0 & \ldots & 0 & & & &  \end{array}\right),
} \]
and for odd $\ell$ it is
\[ {\tiny \left(\begin{array}{rrrr|rrrrr|rrrr}
1 & 1 & 1 & 1 & 2 & 2 & 2 & \ldots & 2 & \sqrt{\ell} & \sqrt{\ell}& \sqrt{\ell}& \sqrt{\ell} \\ 
1 & 1 & 1 & 1 & 2 & 2 & 2 & \ldots & 2 & -\sqrt{\ell} & -\sqrt{\ell}& -\sqrt{\ell}& -\sqrt{\ell} \\ 
1 & 1 & -1 & -1 & -2 & 2 & -2 & \ldots & 2 & \ii^\ell\sqrt{\ell}& -\ii^\ell\sqrt{\ell}& -\ii^\ell\sqrt{\ell}& \ii^\ell\sqrt{\ell} \\ 
1 & 1 & -1 & -1 & -2 & 2 & -2 & \ldots & 2 & -\ii^\ell\sqrt{\ell} & \ii^\ell\sqrt{\ell}& \ii^\ell\sqrt{\ell}& -\ii^\ell\sqrt{\ell} \\  \hline 
1 & 1 & -1 & -1 &  &  &  &  &  & 0 & 0 & 0 & 0 \\
1 & 1 & 1 & 1 &  &  & \vdots &  &  & 0 & 0 & 0 & 0  \\
1 & 1 & -1 & -1 &  & \ldots & c(ij) & \ldots &  & 0 & 0 & 0 & 0 \\
\vdots & \vdots & \vdots & \vdots &  &  & \vdots &  &  & \vdots & \vdots & \vdots & \vdots  \\
1 & 1 & 1 & 1 &  &  &  &  &  & 0 & 0 & 0 & 0 \\ \hline 
1 & -1 & \ii^\ell & -\ii^\ell & 0 & 0 & 0 & \ldots & 0 &  & & & \\ 
1 & -1 & -\ii^\ell & \ii^\ell & 0 & 0 & 0 & \ldots & 0 & & W_{\ell} & & \\ 
1 & -1 & -\ii^\ell & \ii^\ell & 0 & 0 & 0 & \ldots & 0 & & & & \\ 
1 & -1 & \ii^\ell & -\ii^\ell & 0 & 0 & 0 & \ldots & 0 & & & &  \end{array}\right),
} \]
where the specific form of $W_\ell$ is given by Proposition \ref{HH}.
\end{theor}

Recall that $s_{\lambda,\mu} = \dfrac{S_{\lambda,\mu}}{S_{\lambda,\nu_0}}$,
$S=S^T$ and $S_{\lambda,\nu_0}>0$. So the complete $S$-matrix is given by the last
theorem and by:
\begin{propo}
The first row (column) of the $S$-matrix of type $D^{(1)}_\ell$ and level $2$ is
\[ \frac{1}{2\sqrt{2\ell}}
(1,1,1,1,2,\ldots,2,\sqrt{\ell},\sqrt{\ell},\sqrt{\ell},\sqrt{\ell}). \]
The $T$-matrix is
\[ T = \zeta_{24(\ell-1)}^{-\ell(\ell+1)(2\ell+1)}
\diag\left(\zeta_{4\ell}^{(\lambda\mid\lambda)}\right)_{\lambda\in P_\rho^+}. \]
\end{propo}
\begin{proof}
The first row of $S$ has positive real entries and has norm $1$.
The normalization factor for the first row of $s$ is
$\sum_{\lambda\in P_\rho^+} s_{\nu_0,\lambda} \BAR{s_{\nu_0,\lambda}} = 8\ell$.
The $T$-matrix is given in \cite[13.8]{vK}.
\end{proof}

\section{Representations of the vertex operator algebra $V_L^+$}
We largely follow the notation of \cite{tAcDhL} and \cite{tA}.  Let $L =\sqrt{2\ell}\, \ZZ$.  Then $L$ is an even lattice of rank 1 with dual lattice $L^\circ = \dfrac{1}{\sqrt{2\ell}}\ZZ$.  The complete list of inequivalent irreducible representations for the vertex operator algebra $V_L^+$ is: 
\bi
\item $[i]$ for $1 \leq i \leq \ell - 1$ (corresponding to elements $\frac{i}{\sqrt{2\ell}} + L$ in the quotient lattice $L^\circ/L$ for which $\frac{2i}{\sqrt{2\ell}} \notin L$); 
\item $[0]^\pm$ and $[\ell]^\pm$ (corresponding to those elements $x+L$ of $L^\circ/L$ for which $2x\in L$);
\item $[\chi_1]^\pm$ and $[\chi_2]^\pm$ (corresponding to twisted $V_L$ modules).
\ei
By abuse of notation, and to simplify the fusion rules, we write 
\begin{equation}
[0] = [0]^+ + [0]^-\ \ \text{and}\ \ [\ell] = [\ell]^+ + [\ell]^-, \label{abuse}
\end{equation} 
and we note that $[i] = [-i] = [2\ell-i]$ for $0 \le i \le \ell$.  (In particular $[\ell + i] = [\ell - i]$.)  Notice that there are $\ell+7$ inequivalent irreducible representations.  

Let $1 \leq i,j \leq \ell-1$ and let $\epsilon, \epsilon_1, \epsilon_2 \in \{\pm\}$ with multiplicative product (i.e. $(-)(-)= (+)$).  For general $\ell$, we have:
\bean
[i] \ot [j] & = & [i+j] + [i-j] \label{voa-ij} \\ 
\left[ 0 \right]^\epsilon \ot [i] & = & [i] \label{voa-0i} \\
\left[ \ell \right]^\epsilon \ot [i] & = & \left[ \ell-i \right] \label{voa-li} \\
\left[ 0 \right]^{\epsilon_1} \ot \left[ m \right]^{\epsilon_2} & = & \left[ m \right]^{\epsilon_1\epsilon_2}\ \textrm{for $[m] \in \{[0],[\ell],[\chi_1],[\chi_2]\}$} \label{voa-0m} \\ 
\left[\chi_1 \right]^{\epsilon} \ot \left[i \right] & = & \left\{  \begin{array}{cl} \left[\chi_1 \right]^+ + \left[\chi_1 \right]^- & \textrm{for $i$ even} \\ \left[\chi_2 \right]^+ + \left[\chi_2 \right]^- &  \textrm{for $i$ odd} \end{array} \right. \label{voa-chi1i}  \\
\left[\chi_2 \right]^{\epsilon} \ot \left[i \right] & = & \left\{  \begin{array}{cl} \left[\chi_2 \right]^+ + \left[\chi_2 \right]^- & \textrm{for $i$ even} \\ \left[\chi_1 \right]^+ + \left[\chi_1 \right]^- &  \textrm{for $i$ odd} \end{array} \right.  \label{voa-chi2i} 
\eean
So $[0]^+$ is the identity of the fusion algebra.  The remaining fusion rules depend on whether $\ell$ is even or odd.  

For even $\ell$, we have:
\bean
\left[\ell \right]^{\epsilon_1} \ot \left[\ell \right]^{\epsilon_2} & = & \left[0 \right]^{\epsilon_1\epsilon_2} \label{voae-ll} \\
\left[\ell \right]^{\epsilon_1} \ot \left[\chi_1 \right]^{\epsilon_2} & = & \left[\chi_1 \right]^{\epsilon_1\epsilon_2} \label{voae-lchi1} \\
\left[\ell \right]^{\epsilon_1} \ot \left[\chi_2 \right]^{\epsilon_2} & = & \left[\chi_2 \right]^{-\epsilon_1\epsilon_2} \label{voae-lchi2} \\
\hspace{15mm} \left[\chi_1 \right]^{\epsilon_1} \ot \left[\chi_1 \right]^{\epsilon_2} & = & \left[0 \right]^{\epsilon_1\epsilon_2} + \left[\ell \right]^{\epsilon_1\epsilon_2} + [2] + [4] + \ldots + [\ell-2] \label{voae-chi11} \\
\left[\chi_1 \right]^{\epsilon_1} \ot \left[\chi_2 \right]^{\epsilon_2} & = & [1] + [3] + \ldots + [\ell-1] \label{voae-chi12} \\
\left[\chi_2 \right]^{\epsilon_1} \ot \left[\chi_2 \right]^{\epsilon_2} & = & \left[0 \right]^{\epsilon_1\epsilon_2} + \left[\ell \right]^{-\epsilon_1\epsilon_2} + [2] + [4] + \ldots + [\ell-2] \label{voae-chi22} 
\eean
Note that if $\ell$ is even, then all irreducible representations are self-dual.

For odd $\ell$, we have instead:
\bean
\left[\ell \right]^{\epsilon_1} \ot \left[\ell \right]^{\epsilon_2} & = & \left[0 \right]^{-\epsilon_1\epsilon_2} \label{voao-ll} \\
\left[\ell \right]^{\epsilon_1} \ot \left[\chi_1 \right]^{\epsilon_2} & = & \left[\chi_2 \right]^{-\epsilon_1\epsilon_2} \label{voao-lchi1} \\
\left[\ell \right]^{\epsilon_1} \ot \left[\chi_2 \right]^{\epsilon_2} & = & \left[\chi_1 \right]^{\epsilon_1\epsilon_2} \label{voao-lchi2} \\
\left[\chi_1 \right]^{\epsilon_1} \ot \left[\chi_1 \right]^{\epsilon_2} & = & \left[\ell \right]^{\epsilon_1\epsilon_2} + [1] + [3] + \ldots + \left[ \ell-2\right ] \label{voao-chi11} \\
\left[\chi_1 \right]^{\epsilon_1} \ot \left[\chi_2 \right]^{\epsilon_2} & = & \left[ 0 \right]^{\epsilon_1\epsilon_2} + [2] + [4] + \ldots + [\ell-1] \label{voao-chi12} \\
\left[\chi_2 \right]^{\epsilon_1} \ot \left[\chi_2 \right]^{\epsilon_2} & = & \left[\ell \right]^{-\epsilon_1\epsilon_2} + [1] + [3] + \ldots + [\ell-2] \label{voao-chi22} 
\eean
Note that if $\ell$ is odd, then $(\left[\ell\right]^\epsilon)^* = \left[\ell\right]^{-\epsilon}$ and $(\left[\chi_1 \right]^\epsilon)^* = \left[\chi_2\right]^{\epsilon}$.

\section{Isomorphism of fusion algebras}

\begin{theor} \label{mainthm}
Let $L = \sqrt{2\ell}\, \ZZ$ for some $\ell \in \NN$. Then the fusion algebra for the vertex operator
algebra $V_L^+$ is isomorphic to the fusion algebra
\[
\left\{ \begin{array}{cl} D^{(1)}_{\ell,2} & \text{if $\ell \ge 3$} \\
A^{(1)}_{1,2} \ot A^{(1)}_{1,2} & \text{if $\ell = 2$} \\
A^{(1)}_{7,1} \cong \ZZ[\ZZ/8\ZZ] & \text{if $\ell = 1$}.   \end{array} \right.
\] 
(Recall that $D^{(1)}_{3,2} \cong A^{(1)}_{3,2}$.)
\end{theor}

\begin{proof}  The special cases $\ell = 1,2$ were studied in \cite{cDkN} and \cite{cDrGgH}, respectively.  When $\ell = 1$, the fusion rules are generated by $[\chi_1]^+$, say, which has order 8 in the fusion algebra.  When $\ell = 2$, then $V_L^+$ is isomorphic to $L(\frac12,0) \ot L(\frac12,0)$, and $L(\frac12,0)$ (which is the irreducible highest weight module for the Virasoro algebra with central charge $\frac12$ and highest weight $0$) has isomorphic fusion with $A^{(1)}_{1,2}$.  

Now let $\ell \geq 3$.  To prove the desired isomorphism of fusion algebras, we will list the irreducible VOA modules in the following order: 
$$[0]^+,
[0]^-,
[\ell]^+,
[\ell]^-,
[1],
\ldots,
[\ell-1],
[\chi_2]^+,
[\chi_1]^+,
[\chi_1]^-,
[\chi_2]^-.$$
Now we simply check that the fusion rules are satisfied by the corresponding columns of $s$ under componentwise multiplication, implying that $s$ is also a valid $s$-matrix for the fusion algebra of $V_L^+$ by (\ref{verl}).

Equations (\ref{voa-0i}) and (\ref{voa-0m}) are clear by inspection for $\ell$ even or odd, due to the structure of $W_\ell$ in each case.   The cosine identity $c((\ell-i)j) = (-1)^j c(ij)$ explains (\ref{voa-li}) for all $\ell$, when one takes into account the parity of $\ell$.  Similarly, the cosine identity 
$c(\alpha)c(\beta) = c(\alpha+\beta) + c(\alpha-\beta)$,
together with the notational simplifications of (\ref{abuse}), implies (\ref{voa-ij}) for all $\ell$.  (Note that (\ref{abuse}) implies that the last four entries of the column representing $[i]$ are zero for $0 \leq i \leq \ell$.)

Consider (\ref{voa-chi1i}).  When multiplying the corresponding column entries, only the first four will be nonzero.  By inspection, the product is equal to $[\chi_1]^+ + [\chi_1]^-$ if $i$ is even and $[\chi_2]^+ + [\chi_2]^-$ if $i$ is odd, due in part to the structure of $W_\ell$.  Equation (\ref{voa-chi2i}) is similar.  

This completes the portion of the proof that holds for general $\ell$.

Now let $\ell$ be a multiple of 4 (and so $\ii^\ell = 1$).  Then (\ref{voae-ll}) implies that the first four representations generate a subring isomorphic to $\ZZ[\ZZ/2\ZZ \times \ZZ/2\ZZ]$, consistent with (the last four entries in) the first four columns of $s$.  Equations (\ref{voae-lchi1}) and (\ref{voae-lchi2}) follow directly from the specific form of $W_\ell$ in this case. 

Consider (\ref{voae-chi12}).  Only the first four entries of the product of the columns are different from zero, and they are $(\ell,\ell,-\ell,-\ell)$.  Compare this with the sum of the odd columns $[1] + [3] + \ldots + [\ell-1]$.  Clearly the first four and last four entries match up.  Also, the entry for the row corresponding to $[i]$ is
$c(i) + c(3i) + \ldots + c((\ell-1)i) = \frac{s(\ell i)c(\frac{\ell i}{2})}{s(i)} = 0$
by (\ref{cosphi}).

For the left hand side of equation (\ref{voae-chi11}), the first four entries are $\ell$, followed by $\ell-1$ zeroes, and finally followed by $\epsilon_1\epsilon_2(0,2,2,0)$.  For the right hand side, each of the first four entries is also $\ell$.  The last four entries of the sum of the columns for $[0]^{\epsilon_1\epsilon_2}$ and $[\ell]^{\epsilon_1\epsilon_2}$ (and $[2] + \ldots + [\ell-2]$) are $\epsilon_1\epsilon_2(0,2,2,0)$.  Consider now the entry in the row for $[i]$:
$1 + (-1)^i + c(2i) + \ldots + c((\ell-2)i) = (-1)^i + \frac{s((\ell-1)i)}{s(i)} = 0$
by (\ref{DK}).

The argument proving (\ref{voae-chi22}) is similar, only now the last four entries of the column representing the left hand side are $\epsilon_1\epsilon_2(2,0,0,2)$.  So the only change is that now we need $[0]^{\epsilon_1\epsilon_2} + [\ell]^{-\epsilon_1\epsilon_2} + [2] + \ldots + [\ell-2]$.

The case when $4 | \ell$ is complete.  Suppose now that $\ell \equiv 2$ mod 4.  In this case, the general arguments remain the same, but now $\ii^\ell=-1$ and $W_\ell$ has changed.   The reader can check that (\ref{voae-ll}), (\ref{voae-lchi1}), (\ref{voae-lchi2}) and (\ref{voae-chi12}) remain valid in this case.  Moreover, similar arguments as in the $\ell \equiv 0$ mod 4 case verify (\ref{voae-chi11}) and (\ref{voae-chi22}).

Let $\ell \equiv 1$ mod 4 (and so $\ii^\ell = \ii$).  Then (\ref{voao-ll}) implies that the first four representations generate a subring isomorphic to $\ZZ[\ZZ/4\ZZ]$, consistent with (the last four entries in) the first four columns of $s$.

Consider one case of (\ref{voao-lchi1}), e.g. $[\ell]^+ \ot [\chi_1]^+$.  Certainly the first $\ell-3$ entries are consistent with $[\chi_2]^-$.  The last four entries of the left hand side are $(\ii \zeta_8, -\ii \zeta_8^7, -\ii \zeta_8^3, \ii \zeta_8^5) = (\zeta_8^3,\zeta_8^5,\zeta_8,\zeta_8^7)$, which corresponds to the last column of $W_\ell$, i.e. $[\chi_2]^-$.  Notice that changing the sign of $\epsilon$ in $[\ell]^\epsilon$, $[\chi_1]^\epsilon$, or $[\chi_2]^\epsilon$ changes the sign of only the last four entries of each column.  So (\ref{voao-lchi1}) follows.  Similarly, to show (\ref{voao-lchi2}), we need only verify one instance.
Consider the last four entries of $[\ell]^+ \ot [\chi_2]^+$, which are $(\ii \zeta_8^7,-\ii \zeta_8,-\ii\zeta_8^5,\ii \zeta_8^3) = (\zeta_8,\zeta_8^7,\zeta_8^3,\zeta_8^5)$,
consistent with $[\chi_1]^+$.

Now consider (\ref{voao-chi11}).  The left hand side begins $(\ell,\ell,-\ell,-\ell)$, followed by $\ell-1$ zeroes, and then $\epsilon_1\epsilon_2(\ii,-\ii,-\ii,\ii)$.  The first four entries and the last four entries are consistent with $[\ell]^{\epsilon_1\epsilon_2} + [1] + \ldots + [\ell-2]$.  The entry in the row for $[i]$ is
$(-1)^i + c(i) + c(3i) + \ldots + c((\ell-2)i) =  (-1)^i + \frac{s((\ell-1)i)}{s(i)} = 0$
by (\ref{cosphi}) and other trigonometric identities.  Similarly, the left hand side of (\ref{voao-chi22}) has the same first $\ell-3$ entries as above but has $\epsilon_1\epsilon_2(-\ii,\ii,\ii,-\ii)$ as its last four entries, which is consistent with $[\ell]^{-\epsilon_1\epsilon_2} + [1] + \ldots + [\ell-2]$.

The left hand side of (\ref{voao-chi12}) begins $(\ell,\ell,\ell,\ell)$, followed by $\ell-1$ zeroes, and then $\epsilon_1\epsilon_2(1,1,1,1)$.  Clearly, this matches the first four entries and the last four entries of the right hand side.  The entry in the row corresponding to $[i]$ is
$1 + c(2i) + c(4i) + \ldots + c((\ell-1)i) = \frac{s(\ell i)}{s(i)} = 0$
by (\ref{DK}).

This concludes the proof when $\ell \equiv 1$ mod 4.  But when $\ell \equiv 3$ mod 4, then we essentially are taking the complex conjugate of the $s$-matrix in the $\ell \equiv 1$ mod 4 case.  So all the arguments above remain valid, \emph{mutatis mutandis}, as complex conjugation is $\sim$ and hence a fusion algebra morphism.
\end{proof}

\bibliographystyle{amsplain}
\bibliography{references}

\providecommand{\bysame}{\leavevmode\hbox to3em{\hrulefill}\thinspace}
\providecommand{\MR}{\relax\ifhmode\unskip\space\fi MR }
\providecommand{\MRhref}[2]{%
  \href{http://www.ams.org/mathscinet-getitem?mr=#1}{#2}
}
\providecommand{\href}[2]{#2}
\begin{thebibliography}{10}

\bibitem{tA}
Toshiyuki Abe, \emph{Fusion rules for the charge conjugation orbifold}, J.
  Algebra \textbf{242} (2001), no.~2, 624--655.

\bibitem{tAcDhL}
Toshiyuki Abe, Chongying Dong, and Haisheng Li, \emph{Fusion rules for the
  vertex operator algebra {$M(1)$} and {$V\sp +\sb L$}}, Comm. Math. Phys.
  \textbf{253} (2005), no.~1, 171--219.

\bibitem{mC1}
Michael Cuntz, \emph{{F}usion algebras for imprimitive complex reflection
  groups}, J. Algebra \textbf{311} (2007), no.~1, 251--267.

\bibitem{mC2}
\bysame, \emph{Integral modular data and congruences}, J. Algebraic
  Combinatorics, http://www.springerlink.com/content/43l0k1v7w33555k8 (2008).

\bibitem{cDrGgH}
Chongying Dong, Robert~L. Griess, Jr., and Gerald Hoehn, \emph{Framed vertex
  operator algebras, codes and the moonshine module}, Comm. Math. Phys.
  \textbf{193} (1998), no.~2, 407--448.

\bibitem{cDkN}
Chongying Dong and Kiyokazu Nagatomo, \emph{Representations of vertex operator
  algebra {$V_L^+$} for rank one lattice {$L$}}, Comm. Math. Phys. \textbf{202}
  (1999), no.~1, 169--195.

\bibitem{aFsF}
Alex~J. Feingold and Stefan Fredenhagen, \emph{A new perspective on the
  {F}renkel-{Z}hu fusion rule theorem}, J. Algebra \textbf{320} (2008), no.~5,
  2079--2100.

\bibitem{tG}
Terry Gannon, \emph{The automorphisms of affine fusion rings}, Adv. Math.
  \textbf{165} (2002), no.~2, 165--193.

\bibitem{tG2}
\bysame, \emph{Modular data: the algebraic combinatorics of conformal field
  theory}, J. Algebraic Combin. \textbf{22} (2005), no.~2, 211--250.

\bibitem{jH}
James~E. Humphreys, \emph{Introduction to {L}ie algebras and representation
  theory}, Springer-Verlag, New York, 1972, Graduate Texts in Mathematics, Vol.
  9.

\bibitem{vK}
Victor~G. Kac, \emph{Infinite-dimensional {L}ie algebras}, third ed., Cambridge
  University Press, Cambridge, 1990.

\end{thebibliography}

\end{document}